\documentclass[12pt,reqno]{amsart}
\usepackage{a4}
\usepackage{amsmath, amssymb, amsthm}
\usepackage{graphicx, color}
\usepackage{enumerate}
\begin{document}

 \bibliographystyle{plain}
 \newtheorem{theorem}{Theorem}
 \newtheorem{lemma}[theorem]{Lemma}
 \newtheorem{corollary}[theorem]{Corollary}
 \newtheorem{problem}[theorem]{Problem}
 \newtheorem{conjecture}[theorem]{Conjecture}
 \newtheorem{definition}[theorem]{Definition}
 \newtheorem{prop}[theorem]{Proposition}
 \numberwithin{equation}{section}
 \numberwithin{theorem}{section}

 \newcommand{\mo}{~\mathrm{mod}~}
 \newcommand{\mc}{\mathcal}
 \newcommand{\rar}{\rightarrow}
 \newcommand{\Rar}{\Rightarrow}
 \newcommand{\lar}{\leftarrow}
 \newcommand{\lrar}{\leftrightarrow}
 \newcommand{\Lrar}{\Leftrightarrow}
 \newcommand{\zpz}{\mathbb{Z}/p\mathbb{Z}}
 \newcommand{\mbb}{\mathbb}
 \newcommand{\B}{\mc{B}}
 \newcommand{\cc}{\mc{C}}
 \newcommand{\D}{\mc{D}}
 \newcommand{\E}{\mc{E}}
 \newcommand{\F}{\mathbb{F}}
 \newcommand{\G}{\mc{G}}
  \newcommand{\ZG}{\Z (G)}
 \newcommand{\FN}{\F_n}
 \newcommand{\I}{\mc{I}}
 \newcommand{\J}{\mc{J}}
 \newcommand{\M}{\mc{M}}
 \newcommand{\nn}{\mc{N}}
 \newcommand{\qq}{\mc{Q}}
 \newcommand{\PP}{\mc{P}}
 \newcommand{\U}{\mc{U}}
 \newcommand{\X}{\mc{X}}
 \newcommand{\Y}{\mc{Y}}
 \newcommand{\itQ}{\mc{Q}}
 \newcommand{\sgn}{\mathrm{sgn}}
 \newcommand{\C}{\mathbb{C}}
 \newcommand{\R}{\mathbb{R}}
 \newcommand{\T}{\mathbb{T}}
 \newcommand{\N}{\mathbb{N}}
 \newcommand{\Q}{\mathbb{Q}}
 \newcommand{\Z}{\mathbb{Z}}
 \newcommand{\A}{\mathbb{A}}
 \newcommand{\ff}{\mathfrak F}
 \newcommand{\fb}{f_{\beta}}
 \newcommand{\fg}{f_{\gamma}}
 \newcommand{\gb}{g_{\beta}}
 \newcommand{\vphi}{\varphi}
 \newcommand{\whXq}{\widehat{X}_q(0)}
 \newcommand{\Xnn}{g_{n,N}}
 \newcommand{\lf}{\left\lfloor}
 \newcommand{\rf}{\right\rfloor}
 \newcommand{\lQx}{L_Q(x)}
 \newcommand{\lQQ}{\frac{\lQx}{Q}}
 \newcommand{\rQx}{R_Q(x)}
 \newcommand{\rQQ}{\frac{\rQx}{Q}}
 \newcommand{\elQ}{\ell_Q(\alpha )}
 \newcommand{\oa}{\overline{a}}
 \newcommand{\oI}{\overline{I}}
 \newcommand{\dx}{\text{\rm d}x}
 \newcommand{\dy}{\text{\rm d}y}
\newcommand{\cal}[1]{\mathcal{#1}}
\newcommand{\cH}{{\cal H}}
\newcommand{\diam}{\operatorname{diam}}
\newcommand{\bx}{\mathbf{x}}
\newcommand{\Ps}{\varphi}

\parskip=0.5ex

\title[Bounded remainder sets for rotations on the adelic torus]{Bounded remainder sets for\\ rotations on the adelic torus}
\author{Joanna~Furno, Alan~Haynes, Henna~Koivusalo}

\thanks{HK: Work carried out on visit to the University of Houston, supported by the  V\"ais\"al\"a fund.\\
\phantom{A..}MSC 2010: 11J61, 11K38, 37A45}

\allowdisplaybreaks

\begin{abstract}
In this paper we give an explicit construction of bounded remainder sets of all possible volumes, for any irrational rotation on the adelic torus $\A/\Q$. Our construction involves ideas from dynamical systems and harmonic analysis on the adeles, as well as a geometric argument which originated in the study of deformation properties of mathematical quasicrystals.
\end{abstract}

\maketitle

\section{Introduction}
Let $G$ be a compact, metrizable, Abelian group, written additively. There is a unique Haar probability measure on $G$ and, for each $\alpha\in G$, the measure-preserving map $T_\alpha: G\rar\ G$ defined by $T_\alpha(x)=x+\alpha$ is referred to as rotation by $\alpha$ on $G$. It is well known (see \cite[Theorem 4.14]{EinsWard2011}) that the map $T_\alpha$ will be ergodic, and also uniquely ergodic, if and only if the collection of points $\{T^n(0)\}$ is dense in $G$. Another equivalent condition is that, for any Borel set $A\subseteq G$ with boundary of measure $0$ and for every $x\in G$, we have that
\begin{equation}
  \lim_{N\rar\infty}\left(\frac{1}{N}\sum_{n=0}^{N-1}\chi_A(x+n\alpha)-|A|\right)=0,
\end{equation}
where $|A|$ denotes the Haar measure of $A$ (the equivalence of these conditions follows from the proof of \cite[Theorem 6.19]{Walt1982} together with Remark 3 after \cite[Theorem 6.4]{Walt1982}).

For many applications, we would like to know something about the rate of convergence of the above limit, in the case when $T_\alpha$ is ergodic. This is the topic of the well studied field of discrepancy theory, for which we refer the reader to \cite{DrmoTich1997} and \cite{KuipNied1974}. In this paper we are going to study the specific problem of constructing measurable sets $A$ with the property that there exists a constant $C=C(A)>0$ such that, for almost every $x\in G$ and for any $N\in\N$,
\begin{equation}
  \left|\sum_{n=0}^{N-1}\chi_A(x+n\alpha)-N|A|\right|\le C.
\end{equation}
Such sets $A$ are called bounded remainder sets (BRS's) for $T_\alpha$. As in previous work on BRS's, we follow the convention of allowing them to be multisets, i.e. so that $\chi_A$ is allowed to be a finite sum of indicator functions of measurable sets.

Bounded remainder sets for rotations on the groups $\T^s=\R^s/\Z^s, s\in\N,$ have been extensively studied since the 1920's. For $s=1$ the first results were obtained by Hecke \cite{Heck1922}, Ostrowski \cite{Ostr1927/30}, and Kesten \cite{Kest1966/67}, who proved that, for an irrational rotation by $\alpha$ on the circle $\R/\Z$, an interval $A$ will be a BRS if and only if
\begin{equation}
|A|\in\alpha\Z+\Z.
\end{equation}
In the $s=2$ case, Sz\"{u}sz \cite{Szus1954} constructed infinite families of BRS parallelograms. The $s\ge 2$ case was subsequently studied by Liardet \cite{Liar1987}, Rauzy \cite{Rauz1972}, Ferenczi \cite{Fere1992}, Oren \cite{Oren1982}, and Zhuravlev \cite{Zhur2005,Zhur2011,Zhur2012}, as well as other authors. A landmark result was recently attained by Grepstad and Lev \cite{GrepLev2015}, who demonstrated that, for any $s$ and for any $\alpha$ for which $T_\alpha$ is ergodic, the set of all volumes of BRS's for $\alpha$ is
\begin{equation}
  \{n\cdot\alpha +m\ge 0:n\in\Z^s,m\in\Z\}.
\end{equation}
Grepstad and Lev's results also provide examples of parallelotope BRS's of all possible volumes, and they thereby effectively complete the classification of volumes of BRS's for toral rotations in any dimension.

Returning to the level of generality with which we began this section, it is well known that there are uncountably many topological group isomorphism classes of connected, compact, metrizable, Abelian groups. The goal of this paper is to extend the classification of BRS's from the countable family $\T^s, s\in\N$, to an uncountable collection of connected, compact subgroups of the adelic torus $\A/\Q$. Our first result is the following theorem.
\begin{theorem}\label{thm.BRSforA/Q}
  Suppose that $\alpha=(\alpha_\infty,\alpha_2,\alpha_3,\ldots)\in\A/\Q$  and that $\alpha_\infty\notin\Q$. Then the collection of all volumes of BRS's for $T_\alpha$ is
  \begin{equation}\label{eqn.SetOfVols}
    \left\{-\gamma\alpha_\infty+\sum_p\{\gamma\alpha_p\}_p+n\ge 0:\gamma\in\Q,n\in\Z\right\},
  \end{equation}
  where $\{\cdot\}_p:\Q_p\rar\R$ is the $p$-adic fractional part.
\end{theorem}
Note that the requirement that $\alpha_\infty\not\in\Q$ is equivalent to the condition that $\{n\alpha\}_{n\in\N}$ is dense in $\A/\Q$ (we will verify this claim in the next section). Our proof of Theorem \ref{thm.BRSforA/Q} uses a connection between BRS's and deformation properties of mathematical quasicrystals. This connection was observed indirectly by physicists Duneau and Oguey in \cite{DuneOgue90}, and later developed by two of the authors together with Michael Kelly in \cite{HaynKoiv2016,HaynKellKoiv2017} to give a concise geometric proof of the above mentioned result of Grepstad and Lev.
% see \cite[Chapter 4, Corollary 1.2]{KuipNied1974}, as well as the results in the next section)

As indicated in our discussion above, we can generalize Theorem \ref{thm.BRSforA/Q} as follows. Let $\mc{P}$ denote the collection of all prime numbers. Suppose $\mc{Q}$ is a (finite or infinite) subset of $\mc{P}$, write $\mc{Q}=\{p_1,p_2,\ldots\}$, and let $\A_\mc{Q}$ be the projection of $\A$ onto the coordinates indexed by the infinite place and the elements of $\mc{Q}$. We take $\A_\mc{Q}$ with the usual topology, which is also the final topology with respect to this projection. The additive group $\Gamma_\mc{Q}=\Z[1/p_1,1/p_2,\ldots ]$ can be embedded diagonally in $\A_\mc{Q}$, and we identify it with its image under this embedding. We then define
\begin{equation}
  X_\mc{Q}=\A_\mc{Q}/\Gamma_\mc{Q},
\end{equation}
which is easily seen to be a connected, compact, metrizable, Abelian group. Alternatively, $X_\mc{Q}$ is isomorphic as a topological group to the quotient of $\A/\Q$ by the subgroup consisting of cosets $\alpha+\Q$ with the property that $\alpha_\infty=0,\alpha_p\in\Z_p$ for all $p\in\mc{P}$, and $\alpha_p=0$ whenever $p\notin \mc{Q}$. We will denote elements of $\A_\mc{Q}$ and of $X_\mc{Q}$ by $(x_\infty,x_{p_1},x_{p_2},\ldots)$ and, for $\alpha\in X_\mc{Q}$, we write $T_\alpha:X_\mc{Q}\rar X_\mc{Q}$ for the transformation which maps $x$ to $x+\alpha$.
\begin{theorem}\label{thm.BRSforSolenoids}
    Suppose that $\mc{Q}\subseteq\mc{P}$ and that $X_\mc{Q}$ is defined as above. Further, suppose that $\alpha=(\alpha_\infty,\alpha_{p_1},\alpha_{p_2},\ldots)\in X_\mc{Q}$  and that $\alpha_\infty\notin\Q$. Then the collection of all volumes of BRS's for $T_\alpha$ is
  \begin{equation}\label{eqn.SetOfVols2}
    \left\{-\gamma\alpha_\infty+\sum_{p\in\mc{Q}}\{\gamma\alpha_p\}_p+n\ge 0:\gamma\in\Gamma_\mc{Q},n\in\Z\right\},
  \end{equation}
  where $\{\cdot\}_p:\Q_p\rar\R$ is the $p$-adic fractional part.
\end{theorem}
Our strategy will be to prove Theorem \ref{thm.BRSforSolenoids} first in the case when $\mc{Q}$ is a finite set of primes. From this we will then deduce the full statement of the theorem (and hence, of Theorem \ref{thm.BRSforA/Q}). Our method of proof should be flexible enough to give a classification of BRS's for rotations on $(\A/\Q)^s$, for $s\ge 2$. The notation required for the proof of Theorem \ref{thm.BRSforSolenoids} is already complex, so we will not attempt to establish a higher dimensional version in this paper.

This paper is organized as follows. In Section \ref{sec.Back} we will establish notation and review basic background material which we will need for our proofs. In Section \ref{sec.RestVol} we will show how a dynamical argument together with Fourier analysis on the adeles can be used to restrict the possible volumes of BRS's for a given rotation to a countable set. In Section \ref{sec.ConstVol}, which contains the bulk of the proof, we employ a cut and project set construction to give examples of BRS's of all possible allowable volumes.

\section{Background and Notation}\label{sec.Back}
In this section we introduce our notation and give a brief review of the results from $p$-adic analysis which we will use. For more details concerning background material the reader is encouraged to consult \cite[Chapter 1]{Kobl1984} and \cite[Chapter 3]{RamaVale1999}.

In all of what follows we use the symbols $p$ and $q$ to denote prime numbers. For each prime $p$, we write $|\cdot |_p$ for the usual $p$-adic absolute value on $\Q$, and $\Q_p$ for the field of $p$-adic numbers, the completion of $\Q$ with respect to $|\cdot|_p$. The absolute value $|\cdot|_p$ extends in the natural way from $\Q$ to $\Q_p$, and we write $\Z_p$ for the ring of $p$-adic integers, the subring of elements $x\in\Q_p$ with $|x|_p\le 1$. To avoid confusion with other metrics, and with our notation for Haar measure, we write $|\cdot|_\infty$ for the usual Archimedean absolute value on $\R$.

Every element $x\in\Q_p$ has a unique expansion of the form
\begin{equation}
  x=\sum_{i=-N}^{-1}b_ip^i+\sum_{i=0}^\infty b_ip^i,
\end{equation}
with $N\ge 0$ and $b_i\in\{0,1,\ldots , p-1\}$ for each $i$. Given this expansion, we define the $p$-adic fractional part of $x$ to be the rational number
\begin{equation}
  \{x\}_p=\sum_{i=-N}^{-1}b_ip^i,
\end{equation}
with the usual convention that the empty sum takes the value $0$.

It is well known that $\Q_p$ is a totally disconnected, locally compact field, with Pontryagin dual $\widehat{\Q}_p\cong\Q_p$. To make this more precise, for each $y\in\Q_p$ the function $\psi_y:\Q_p\rar\C$ defined by
\begin{equation}
  \psi_y(x)=e(\{yx\}_p),
\end{equation}
where $e(z)=e^{2\pi i z}$, is a continuous character on the additive group of $\Q_p$, and the map $y\mapsto \psi_y$ is a topological group isomorphism from $\Q_p$ to $\widehat{\Q}_p$.

%In order to study dynamics in a non-trivial finite measure space, we would like to form the quotient of $\Q_p$ by a proper, closed, co-compact subgroup. It turns out that this is not possible, because any subgroup of $\Q_p$ must be dense in a ball centered at $0$. However we can remedy the situation with a minor modification. The additive group of $\Z[1/p]$ embeds diagonally in

Now suppose that $\mc{Q}=\{p_1,p_2,\ldots\}$ is a subset of $\mc{P}$. As a set, $\A_\mc{Q}$ is defined to be the collection of elements
\begin{equation}
  \alpha=(\alpha_\infty,\alpha_{p_1},\alpha_{p_2},\ldots)\in\R\times\prod_{p\in\mc{Q}}\Q_p
\end{equation}
which satisfy $\alpha_p\in\Z_p$ for all but finitely many primes $p\in\mc{Q}$. It follows from the strong triangle inequality for the non-Archimedean absolute values that the elements of $\A_\mc{Q}$ form a ring under pointwise addition and multiplication. A natural topology on $\A_\mc{Q}$ is the restricted product topology, with respect to the sets $\Z_p$, for $p\in\mc{Q}$. With this topology the additive group of $\A_\mc{Q}$ is a locally compact topological group. This group is also self dual, with an explicit isomorphism given by the map from $\A_\mc{Q}$ to $\widehat{\A}_\mc{Q}$ defined by $\beta\mapsto\psi_\beta^\mc{Q},$ where
\begin{equation}
  \psi_\beta^\mc{Q}(\alpha)=e(-\beta_\infty\alpha_\infty)\cdot\prod_{p\in\mc{Q}}e(\{\beta_p\alpha_p\}_p).
\end{equation}
Note that in this product, all but finitely many terms are equal to $1$.

The group $\Gamma_\mc{Q}=\Z[1/p_1,1/p_2,\ldots]$ embeds in $\A_\mc{Q}$ by the map $\gamma\mapsto (\gamma,\gamma,\gamma\ldots)$, and we identify it with its image under this map. With this identification, it is easy to check that $\Gamma_\mc{Q}$ is a discrete and closed subgroup of $\A_\mc{Q}$, and that a strict fundamental domain for the quotient group $X_\mc{Q}$ is given by the collection of points
\begin{equation}
  [0,1)\times\prod_{p\in\mc{Q}}\Z_p\subseteq\A_\mc{Q}.
\end{equation}
Therefore the group $X_\mc{Q}$ is compact, and its dual group is the subset of characters on $\A_\mc{Q}$ which are trivial on $\Gamma_\mc{Q}$. More explicitly, the map from the discrete group $\Gamma_\mc{Q}$ to $\widehat{X}_S$ given by $\gamma\mapsto \psi_\gamma^\mc{Q}$ is an isomorphism of topological groups.

To conclude this preliminary section we explain why the hypothesis in Theorem \ref{thm.BRSforSolenoids} that $\alpha_\infty\notin\Q$ is necessary and sufficient to guarantee that the collection of points $\{n\alpha\}_{n\in\N}$ is dense (and in fact that the map $T_\alpha$ is uniquely ergodic) in $X_\mc{Q}$. It is easy to see that this hypothesis is necessary for density, since otherwise the real component of $n\alpha$ will always lie in a finite set. For the other direction we use the following general form of Weyl's criterion.
\begin{lemma}\cite[Chapter 4, Corollary 1.2]{KuipNied1974}
  Suppose that $G$ is a compact Abelian group. A sequence $\{x_n\}_{n\in\N}\subseteq G$ is uniformly distributed in $G$ with respect to Haar measure if and only if, for every nontrivial character $\chi\in\widehat{G}$,
  \begin{equation}
    \lim_{N\rar\infty}\frac{1}{N}\sum_{n=1}^N\chi(x_n)=0.
  \end{equation}
\end{lemma}
If $\alpha_\infty\notin\Q$ and $\gamma\in\Gamma_\mc{Q}\setminus\{0\}$, we have that
\begin{equation}
  \frac{1}{N}\sum_{n=1}^N\psi_\gamma^\mc{Q}(n\alpha)=\frac{1}{N}\sum_{n=1}^Ne\left(n\left(-\gamma\alpha_\infty+\sum_{p\in\mc{Q}}\{\gamma\alpha_p\}_p\right)\right),
\end{equation}
and the right hand side tends to $0$ as $N\rar\infty$. In this calculation we have used the facts that the sum over $p\in\mc{Q}$ contains only finitely many non-zero terms, and also that
\begin{equation}
  \{n\gamma\alpha_p\}_p-n\{\gamma\alpha_p\}_p\in\Z,
\end{equation}
for all $n\in\Z$.

Finally we mention that, for rotations on compact metrizable groups, e.g. $T_\alpha$ on $X_\mc{Q}$, existence of dense orbits and unique ergodicity of the rotation with respect to Haar measure are equivalent (see \cite[Theorem 4.14]{EinsWard2011}). Therefore the condition that $\alpha_\infty\notin\Q$ is necessary and sufficient for $T_\alpha$ to be uniquely ergodic. This is important in what follows, as unique ergodicity implies that the convergence of Birkhoff sums to their ergodic averages is independent of the starting point for the averaging. In our context this means that, when $\alpha_\infty\notin\Q$, translates of BRS's for $T_\alpha$ are also BRS's.

\section{Restriction of Possible Volumes}\label{sec.RestVol}
In this section we will prove that, for any $\alpha$ satisfying the hypotheses of Theorem \ref{thm.BRSforSolenoids}, the set of all possible volumes of BRS's for $T_\alpha$ is a subset of the set in \eqref{eqn.SetOfVols2}. The line of this simple argument is not new and indeed has been used in other contexts by many authors (e.g. \cite{FursKeynShap1973}, \cite[Theorem 2]{Hala1976}, \cite[Section 2.2]{GrepLev2015}, and \cite[Chapter 14]{GottHedl1955}). We only include it for the sake of completeness. First we establish the following elementary lemma.
\begin{lemma}\label{lem.CoBound}
  A measurable set $A\subseteq X_\mc{Q}$ is BRS for $T_\alpha$ if and only if there exists a bounded, measurable function $g:X_\mc{Q}\rar\R$ satisfying
  \begin{equation}\label{eqn.CoBoundEqn}
    \chi_A(x)-|A|=g(x)-g(x+\alpha),
  \end{equation}
  for all $x\in X_\mc{Q}$.
\end{lemma}
\begin{proof}
First of all, suppose that $A$ is any measurable set for which there exists a function $g$ as in the statement of the lemma. Then we have that
\begin{align}
  \left|\sum_{n=0}^{N-1}\chi_A(x+n\alpha)-N|A|\right|=\left|g(x)-g(x+N\alpha)\right|\le 2 \|g\|_\infty,
\end{align}
which shows that $A$ is a BRS for $T_\alpha$.

For the other direction, suppose that $A$ is a measurable BRS for $T_\alpha$ and define $g:X_\mc{Q}\rar\R$ by
\begin{equation}
g(x)=\liminf_{N\rar\infty}\left(\sum_{n=0}^{N-1}\chi_A(x+n\alpha)-N|A|\right).
\end{equation}
Then $g$ is bounded and measurable. Furthermore, for any $N\in\N$ we have that
\begin{equation}
 \chi_A(x)-|A|+\sum_{n=0}^{N-1}\left(\chi_A(x+(n+1)\alpha)-|A|\right)=\sum_{n=0}^N\left(\chi_A(x+n\alpha)-|A|\right),
\end{equation}
and taking the liminf of both sides of the equation, we see that \eqref{eqn.CoBoundEqn} holds.
\end{proof}
Note that Lemma \ref{lem.CoBound} immediately extends to cover the case of multisets, i.e. when $\chi_A$ is replaced by a finite sum of indicator functions of measurable sets. Now suppose that $A\subseteq X_\mc{Q}$ is a measurable BRS for $T_\alpha$, let $g(x)$ be as in the statement of the lemma above, and define $h:X_\mc{Q}\rar\C$ by
\begin{equation}
  h(x)=e(g(x)).
\end{equation}
Then we have that
\begin{equation}
  h(x+\alpha)=e(g(x)-\chi_A(x)+|A|)=e(|A|)h(x).
\end{equation}
Writing the Fourier coefficients of $h$ as $\widehat{h}(\gamma),$ for $\gamma\in\Gamma_\mc{Q}$, we also have that
\begin{align}
  h(x+\alpha)=\sum_{\gamma\in\Gamma_\mc{Q}}\widehat{h}(\gamma)\psi_\gamma^\mc{Q}(\alpha)\psi_\gamma^\mc{Q}(x).
\end{align}
Combining these two formulas, we obtain the expansion
\begin{equation}
  0=\sum_{\gamma\in\Gamma_\mc{Q}}\left(e(|A|)-\psi_\gamma^\mc{Q}(\alpha)\right)\widehat{h}(\gamma)\psi_\gamma^\mc{Q}(x).
\end{equation}
Since $h(x)$ is not the identically zero function, applying the Plancherel theorem gives that
\begin{equation}
  e(|A|)=\psi_\gamma^\mc{Q}(\alpha),
\end{equation}
for some $\gamma\in\Gamma_\mc{Q}$. By the definition of $\psi_\gamma^\mc{Q}$, this shows that $|A|$ must lie in the set \eqref{eqn.SetOfVols2}.

\section{Construction of BRS's of All Allowable Volumes}\label{sec.ConstVol}
In this section we complete the proof of Theorem \ref{thm.BRSforSolenoids}, and therefore also Theorem \ref{thm.BRSforA/Q}, by constructing BRS's of all possible volumes in \eqref{eqn.SetOfVols2}. Our argument proceeds in several steps. First we introduce the cut and project set construction and we show how it can be used to produce a countable family of BRS's for $T_\alpha$. Then we specialize to the case when $\mc{Q}$ is finite, and we show that every volume in the set \eqref{eqn.SetOfVols2} can be realized as a finite sum of volumes of BRS's from our cut and project set construction. Finally, we show how the result of Theorem \ref{thm.BRSforSolenoids}, for the case when $\mc{Q}$ is infinite, can be deduced from the finite case.

\subsection{Cut and project set construction}\label{sec.CutAndProj}
Suppose that $G$ is a locally compact Abelian group and that $E$ and $F$ are complementary subgroups of $G$, so that $E\cap F=\{0\}$ and $G=E+F$. Let $\pi_E:G\rar E$ and $\pi_F:G\rar F$ be the projections onto $E$ and $F$ with respect to this decomposition.

Suppose that $\Gamma$ is a lattice in $G$, i.e. a discrete subgroup for which the quotient $G/\Gamma$ has a measurable fundamental domain of finite volume. Since $G$ is Abelian, this is equivalent to requiring that $\Gamma$ be discrete and co-compact. Let $A\subseteq F$ and set $S=A+E$. The cut and project set $Y=Y(A)\subseteq E$ defined by this data is
\begin{equation}
  Y=\pi_E(\Gamma\cap S).
\end{equation}
To match this with standard terminology in the study of such sets, the group $G$ is the total space, $E$ is the physical space, $F$ is the internal space, $A$ is the window, and $S$ is the strip. A schematic summarizing this construction is provided in Figure 1.

\begin{figure}\label{fig.CutAndProj}
\caption{Definition of a cut and project set.}
\[
  \begin{array}{ccccc}
    E & \xleftarrow{\pi_E} & X & \xrightarrow{\pi_F} & F \\
     \cup &  & \vee &  & \cup \\
     Y&  & \Gamma &  & A \vspace*{.07in}\\
     y& \leftarrow & \gamma & \rar  & y^*\\
  \end{array}\]
\end{figure}

For our application, using the notation in the statement of Theorem \ref{thm.BRSforSolenoids}, we take $G=\A_\mc{Q}\times\A_\mc{Q},$
\begin{equation}
  F=\{(0,x):x\in\A_\mc{Q}\},
\end{equation}
and
\begin{equation}
  E=\{(x,-x\alpha):x\in\A_\mc{Q}\},
\end{equation}
where, by slight abuse of notation, $\alpha$ is taken to be any representative in $\A_\mc{Q}$ for the quantity occurring in the statement of Theorem \ref{thm.BRSforSolenoids}. The subspace $E$ does technically depend on this choice of representative, but the collection of bounded remainder sets which we construct in what follows does not. For the lattice $\Gamma$ we take
\begin{equation}
  \Gamma=\{(\gamma_1,\gamma_2):\gamma_1,\gamma_2\in\Gamma_\mc{Q}\}\cong\Gamma_\mc{Q}\times\Gamma_\mc{Q},
\end{equation}
where, as above, $\Gamma_\mc{Q}=\Z[1/p_1,1/p_2,\ldots].$

We are left with a choice of window $A\subseteq F$. For this, let $\lambda\in\Gamma_\mc{Q}$ and define
\begin{align}
  L&=\{(x,x\lambda):x\in\A_\mc{Q}\}\leqslant G,\\
  \Gamma_1&=\{(\gamma,\gamma\lambda):\gamma\in\Gamma_\mc{Q}\}\leqslant\Gamma,~\text{and}\\
  \Gamma_2&=\{(0,\gamma):\gamma\in\Gamma_\mc{Q}\}\leqslant\Gamma.
\end{align}
Then every element of $\Gamma$ can be written uniquely as an element of $\Gamma_1$ plus an element of $\Gamma_2,$ the group $\Gamma_1$ is a lattice in $L$, and a fundamental domain for $L/\Gamma_1$ is given by
\begin{equation}\label{eqn.LWindow}
W=\left\{(x,x\lambda):x\in[0,1)\times\prod_{p\in\mc{Q}}\Z_p\right\}.
\end{equation}
Finally, we take $A=\pi_F(W)$ and we let $S$ and $Y$ be defined as above.

In order to clarify the connection between the cut and project set we have constructed and the dynamics of $T_\alpha$ on $\A_\mc{Q}$, we first observe that the projections $\pi_E$ and $\pi_F$ are given explicitly as
\begin{align}\label{eqn.Projs}
  \pi_E((x,y))=(x,-x\alpha)\quad\text{and}\quad\pi_F((x,y))=(0,y+x\alpha).
\end{align}
Let us identify $F$ with $\A_\mc{Q}$ (and $A$ with the corresponding subset of $\A_\mc{Q}$), by ignoring the first coordinate. The computations in \eqref{eqn.Projs} show that the cut and project set $Y$ can be described as the set of all points $(\gamma_1,-\gamma_1\alpha),$ with $\gamma_1\in\Gamma_\mc{Q},$ for which there exists a point $\gamma_2\in\Gamma_\mc{Q}$ such that $\gamma_2+\gamma_1\alpha\in A$. If, for a given value of $\gamma_1$, there is more than one such $\gamma_2$, then the corresponding point in the cut and project set occurs with appropriate multiplicity.

To summarize, for any $\gamma_1\in\Gamma_\mc{Q}$, the multiplicity with which the point $(\gamma_1,-\gamma_1\alpha)$ occurs in $Y$ is equal to the value of $\chi_A(\gamma_1\alpha)$, where $\chi_A:X_\mc{Q}\rar\R$ is the indicator function of the projection to $X_\mc{Q}$ of $A$, viewed as a multiset in $X_\mc{Q}$. Recall that, because of unique ergodicity (see comments at the end of Section 2), the quantities that we are ultimately interested in studying are sums of the form
\begin{equation}\label{eqn.IndFuncSum}
  \sum_{n=0}^{N-1}\chi_A(n\alpha),
\end{equation}
and, from our discussion so far, we can now see that these correspond precisely to numbers of points in subsets of $Y$ which lie in bounded regions determined by $N$.

\subsection{Case when $\mc{Q}$ is finite}

Suppose that $\mc{Q}=\{p_1,p_2,\ldots ,p_k\}$ is a finite set of primes, and let $\gamma\in\Gamma_\mc{Q}$ and $n\in\Z$ be chosen so that the number $\xi$ defined by
 \begin{equation}\label{eqn.XiForm}
  \xi=-\gamma\alpha_\infty+\sum_{i=1}^k\{\gamma\alpha_{p_i}\}_{p_i}+n
  \end{equation}
  is non-negative. We now demonstrate how to construct a BRS for $T_\alpha$ of volume $\xi$, and for this purpose we may assume, without loss of generality, that $\gamma\not= 0$.
  
  First we establish that, for each positive integer $\ell$, if $\gamma=\pm1/(p_1\cdots p_k)^\ell$, then there is at least one integer $n$ for which we can construct a bounded remainder set of the above volume. With reference to the formula for $\xi$ in \eqref{eqn.XiForm}, let
\begin{equation}
  \lambda_1=\sum_{i=1}^k\{\gamma\alpha_{p_i}\}_{p_i}+n\quad\text{and}\quad\lambda_2=-\gamma,
\end{equation}
and set $\lambda=\lambda_1/\lambda_2$. Suppose now that $n$ is chosen so that
  \begin{equation}\label{eqn.UnitHyp}
  |\lambda+\alpha_\infty|_\infty\cdot\prod_{i=1}^k|\lambda+\alpha_{p_i}|_{p_i}\not=0.
  \end{equation}
This is always possible because $\alpha_\infty\not\in\Q$, so product on the left of \eqref{eqn.UnitHyp} will vanish if and only if $\lambda=-\alpha_{p_i}$, for some $1\le i\le k$.
  
For every $1\le j\le k$ we have
\begin{equation}
  \lambda_1+\lambda_2\alpha_{p_j}=-\gamma\alpha_{p_j}+\sum_{i=1}^k\{\gamma\alpha_{p_i}\}_{p_i}+n\in\Z_{p_j},
\end{equation}
so it follows that
\begin{equation}
  \prod_{i=1}^k|\lambda_1+\lambda_2\alpha_{p_i}|_{p_i}=\frac{1}{M},
\end{equation}
for some $M\in\N$. Therefore, using the Weil product formula, we have that
\begin{align}\label{eqn.XiVolCalc}
  |\lambda+\alpha_\infty|_\infty\cdot\prod_{i=1}^k|\lambda+\alpha_{p_i}|_{p_i}=\frac{|\lambda_1+\lambda_2\alpha_\infty|_\infty\cdot\prod_{i=1}^k|\lambda_1+\lambda_2\alpha_{p_i}|_{p_i}}{|\lambda_2|_\infty\cdot\prod_{i=1}^k|\lambda_2|_p}=\frac{\xi}{M}.
\end{align}

With reference to the cut and project set constructed in Section \ref{sec.CutAndProj}, we will argue that, for each $\varsigma\in\Gamma_\mc{Q}$, there is exactly one $\varsigma'\in\Gamma_\mc{Q}$ for which
\begin{equation}\label{eqn.LatDecomp}
  (0,\varsigma)+(\varsigma',\varsigma'\lambda)\in S.
\end{equation}
We caution that the usual intuition from Euclidean space (in which the analogous statement is close to obvious) is not entirely applicable in our setting. In order to prove the result here, it is necessary to use the fact that $\lambda+\alpha\in\A_\mc{Q}^*$. Let $\beta=(\lambda+\alpha)^{-1}\in\A_\mc{Q}$ and define a map $\rho:\Gamma_2\rar L$ by
\begin{equation}
  \rho((0,\varsigma))=(\varsigma\beta,\varsigma\beta\lambda).
\end{equation}
Then we have
\begin{equation}
  (0,\varsigma)-\rho((0,\varsigma))=(-\varsigma\beta,\varsigma\beta\alpha)\in E,
\end{equation}
and this implies that
\begin{align}
  S\cap (L+(0,\varsigma))&=(0,\varsigma)+\left((S-(0,\varsigma))\cap L\right)\\
  &=(0,\varsigma)+(W-\rho((0,\varsigma))).
\end{align}
Since $W$ is a fundamental domain for $L/\Gamma_1$ we can now set
\begin{equation}
  (\varsigma',\varsigma'\lambda)=(W-\rho((0,\varsigma)))\cap\Gamma_1,
\end{equation}
and the claim at the beginning of the paragraph follows. Since every element of $\Gamma$ is uniquely expressible as an element of $\Gamma_1$ plus an element of $\Gamma_2,$ this also shows that there is a bijective correspondence between elements of $S\cap\Gamma$ and decompositions of the form \eqref{eqn.LatDecomp}.

Putting all of this together, for each $\varsigma\in\Gamma_\mc{Q}$ let $\varsigma'\in\Gamma_\mc{Q}$ be determined by \eqref{eqn.LatDecomp}. Then the quantity in \eqref{eqn.IndFuncSum} is equal to
\begin{align}
  &\#\{\varsigma\in\Gamma_\mc{Q}:\varsigma'\in\Z,0\le \varsigma'< N\}  \\
  &\qquad =\#\left\{(\Z\cap [0,N))\cap\bigcup_{\varsigma\in\Gamma_\mc{Q}}\left(\varsigma\beta+\left([0,1)\times\prod_{i=1}^k\Z_{p_i}\right)\right)\right\} \\
  &\qquad =\#\left\{\left([0,1)\times\prod_{i=1}^k\Z_{p_i}\right)\cap\bigcup_{\varsigma\in\Gamma_\mc{Q}}\left((\Z\cap [0,N))-\varsigma\beta\right) \right\}\\
  &\qquad=N\cdot|\lambda+\alpha_\infty|_\infty\cdot\prod_{i=1}^k|\lambda+\alpha_{p_i}|_{p_i}+O(1).
\end{align}
This shows that $A$ is a BRS for $T_\alpha$ which, using \eqref{eqn.XiVolCalc}, has volume given by
\begin{equation}\label{eqn.WindowVols}
  |A|=|\lambda+\alpha_\infty|_\infty\cdot\prod_{i=1}^k|\lambda+\alpha_{p_i}|_{p_i}=\frac{\xi}{M}.
\end{equation}
We then take $A'$ to be the multiset obtained as the image under the canonical projection $\tau:\A_\mc{Q}\rar X_\mc{Q}$ of the set
\begin{equation}
\mc{B}=[0,M\cdot |\lambda+\alpha_\infty|_\infty)\times\prod_{i=1}^k \overline{B}(0,|\lambda+\alpha_{p_i}|_{p_i})\subseteq\A_\mc{Q}.
\end{equation}
This set is a BRS for $T_\alpha$ of volume $\xi$.

To summarize, at this point we have succeeded in showing that for each $\gamma$ of the form $\pm1/(p_1\cdots p_k)^\ell$, there is at least one choice of $n$ for which we can construct a BRS for $T_\alpha$ of volume $\xi$. For the general case, first note that, for any $\gamma\in\Gamma_{\mc{Q}}$ and for any integer $m$,
\begin{equation}
m\{\gamma\alpha_{p_i}\}_{p_i}-\{m\gamma\alpha_{p_i}\}_{p_i}\in\Z,
\end{equation}
for each $1\le i\le k$. Now suppose that $\gamma'\in\Gamma_{\mc{Q}}$ and $n'\in\Z$ are chosen so that the number $\xi'$ defined by
\begin{equation}
\xi'=-\gamma'\alpha_\infty+\sum_{i=1}^k\{\gamma'\alpha_{p_i}\}_{p_i}+n'
\end{equation}
is non-negative. Then there must be $\gamma, n,$ and $\xi$ of the special form above, together with integers $m\in\N$ and $m'\in\Z$, with the property that
\begin{equation}
\xi'=m\xi+m'.
\end{equation}

If $m'\ge 0$, then we succeed easily. In that case we may simply take our BRS to be any multiset consisting of $m$ copies of the set $A'$ constructed above, together with $m'$ copies of $X_p$.

For $m'<0$, we would like to take $m$ copies of the image of $A'$ under the projection $\tau$ to $X_p$, and then remove $|m'|$ copies of $X_p$. However, we need to be careful to make sure that each point in our projected set has at least $|m'|$ preimages to remove. Therefore, to justify this step, begin by forming a union of $m$ disjoint copies of $A$ in $F$ to obtain the set
\begin{equation}
\mc{B}'=[0,mM\cdot  |\lambda+\alpha_\infty|_\infty)\times\prod_{i=1}^k \overline{B}(0,|\lambda+\alpha_{p_i}|_{p_i})\subseteq\A_\mc{Q}.
\end{equation}
This set is a Cartesian product of balls in $\R$ and in the spaces $\Q_{p_i}$. Write $z$ for the length of the real interval in this product, and suppose that the product of the volumes of the 
$p_i$-adic balls is $R/S$. Then, if $x$ is any point in $\mc{B}'$, there are at least
\[\left\lfloor \frac{zR}{S}\right\rfloor =\left\lfloor|\mc{B}'|\right\rfloor\]
integers $a$ with the property that (using the diagonal embedding) the point
\[x+\frac{aS}{R}\]
also lies in $\mc{B}'$. All of these points are equivalent modulo $\Gamma_{\mc{Q}}$, which shows that every point in $\tau (\mc{B}')$ has at least $\lfloor |\mc{B}'|\rfloor$ preimages in the original set.  Since $\xi'>0$, the volume of $\mc{B}'$ must be greater than $|m'|$, so this gives the result that we wanted, and it completes the proof of Theorem \ref{thm.BRSforSolenoids} in the special case when $\mc{Q}$ is finite.

\subsection{Case when $\mc{Q}$ is infinite} To prove Theorem \ref{thm.BRSforSolenoids} in the general case, suppose that $\gamma\in\Gamma_\mc{Q}$ and $n\in\Z$ are chosen so that
\begin{equation}
  \xi=-\gamma\alpha_\infty+\sum_{p\in\mc{Q}}\{\gamma\alpha_p\}_p+n\ge 0.
\end{equation}
Let $\mc{Q}'\subset\mc{Q}$ be the set of primes $p$ in $\mc{Q}$ for which $|\alpha_p|_p>1$ or $|\gamma|_p>1$, and let $\alpha'$ be the canonical projection of $\alpha$ to $X_{\mc{Q}'}$. Note that $\gamma\in\Gamma_{\mc{Q}'}$ and that
\begin{equation}
  \{\gamma\alpha_p\}_p=0\quad\text{for all}\quad p\in\mc{Q}\setminus\mc{Q}'.
\end{equation}
Since $\mc{Q}'$ is a finite collection of primes, it follows from what we have already shown that we can construct a BRS $A'\subseteq X_{\mc{Q}'}$ of volume $\xi$ for the rotation $T_{\alpha'}$ on $X_{\mc{Q}'}$. Then, since $\alpha_p\in\Z_p$ whenever $p\in\mc{Q}\setminus\mc{Q}'$, we have that the set
\begin{equation}
  A'\times\prod_{p\in\mc{Q}\setminus\mc{Q}'}\Z_p
\end{equation}
is a BRS of volume $\xi$ for $T_\alpha$. This completes the proofs of our main theorems.

\vspace{.15in}

{\footnotesize
\noindent
JF,~AH: Department of Mathematics, University of Houston,\\
Houston, TX, United States.\\
jfurno@math.uh.edu,~haynes@math.uh.edu\\

\noindent
HK: Faculty of Mathematics, University of Vienna,\\
Vienna, Austria.\\
henna.koivusalo@univie.ac.at
}

\end{document}